\documentclass[11pt]{article}
\usepackage{amsmath}
\usepackage{amsthm}
\usepackage{amsfonts}
\usepackage{amssymb}

\newtheorem{thm}{Theorem}[section]
\newtheorem{lem}[thm]{Lemma}

\theoremstyle{definition}
 
\newcommand{\nequiv}{\not\equiv}

\let\abs=\envert
\let\ep=\epsilon
\newcommand{\floor}[1]{\left\lfloor#1\right\rfloor}

\newcommand{\Z}{{\mathbf{Z}}}
\newcommand{\Q}{{\mathbf{Q}}}
\newcommand{\OK}{{\mathcal K}}
\newcommand{\OO}{{\mathcal O}}
\newcommand{\p}{{\mathfrak p}}
\newcommand{\q}{{\mathfrak q}}

\theoremstyle{remark}

\begin{document}
\title{A new upper bound for odd perfect numbers of a special form
\footnote{This preprint is a revised version of the paper
which appeared in Colloq. Math. \textbf{156} (2019), 15--23.
See the corrigendum attached in this preprint.}
}
\author{Tomohiro Yamada}
\date{}
\maketitle

\begin{abstract}
We give a new effectively computable
upper bound of odd perfect numbers whose Euler factors
are powers of fixed exponent, improving our old result from \cite{Ymd}.
\end{abstract}

\section{Introduction}\label{intro}

As usual, let $\sigma(N)$ denote the sum of divisors of a positive integer $N$.
$N$ is called perfect if $\sigma(N)=2N$.
Though it is not known whether or not an odd perfect number exists,
many conditions which must be satisfied by such a number are known.

Suppose that $N$ is an odd perfect number.
Euler has shown that
\[
N=p^{\alpha} q_1^{2\beta_1}\cdots q_r^{2\beta_r}
\]
for distinct odd primes $p, q_1, \ldots, q_r$ and positive integers
$\alpha, \beta_1, \ldots, \beta_r$ with $p\equiv \alpha\equiv 1\pmod{4}$.

The special case $\beta_1=\beta_2=\cdots =\beta_r=\beta$
has been considered by many authors.
Steuerwald \cite{St} proved that we cannot have $\beta_1=\cdots=\beta_r=1$.
McDaniel \cite{Mc} proved that we cannot have $\beta_1\equiv\cdots\equiv\beta_r\equiv1\pmod{3}$.
If $\beta_1=\cdots=\beta_r=\beta$, then it is known
that $\beta\neq 2$ (Kanold \cite{Ka1}), $\beta\neq 3$ (Hagis
and McDaniel \cite{HMD}), $\beta\neq 5, 12, 17, 24, 62$
(McDaniel and Hagis \cite{MDH}), and $\beta\neq 6, 8, 11, 14, 18$
(Cohen and Williams \cite{CW}).
In their paper \cite{HMD}, Hagis and McDaniel conjectured
that $\beta_1=\cdots=\beta_r=\beta$ does not occur.
We \cite{Ymd} proved that if $\beta_1=\cdots=\beta_r=\beta$,
then $r\leq 4\beta^2+2\beta+2$ and
\[N<2^{4^{4\beta^2+2\beta+3}}.\]
We call this upper bound for $N$ \textit{the classical bound}.

In the RIMS workshop on analytic number theory in 2014,
we have given an improved upper bound for such numbers,
although this result has been published nowhere
(even in the preprint form or in an unrefereed proceedings).
We proved that $r<2\beta^2+O(\beta\log\beta)$
with an effectively computable implicit constant.
There we used the arithmetic in quadratic fields and
lower bounds for linear forms in logarithms.

In this paper, we shall give a slightly stronger result and a simpler proof,
using less arithmetic in quadratic fields and
linear algebraic argument instead of Baker's method.

\begin{thm}\label{th}
If $N=p^{\alpha} q_1^{2\beta}\cdots q_r^{2\beta}$ with $p, q_1, \ldots, q_r$
distinct primes is an odd perfect number,
then $r\leq 2\beta^2+8\beta+2$ and \[N<2^{4^{2\beta^2+8\beta+3}}.\]
Further, the coefficient $8$ of $\beta$ can be replaced by $7$
if $2\beta+1$ is not a prime, or $\beta\geq 29$.
\end{thm}

The upper bound for $N$ immediately follows from the upper bound for $r$
and Nielsen's result \cite{Nie} that if $N$ is an odd perfect number, then $N<2^{4^{\omega(N)}}$,
where $\omega(N)$ denotes the number of distinct prime factors of $N$.

In Section \ref{reduction}, we use a method used in \cite{Ymd}
to reduce the theorem to an upper bound for the number of solutions of some diophantine equations.

\begin{lem}\label{lm0}
Assume that $l=2\beta+1$ is a prime $\geq 19$.
If $N=p^{\alpha} q_1^{2\beta}\cdots q_r^{2\beta}$ with $p, q_1, \ldots, q_r$
distinct primes is an odd perfect number,
then, for each prime $q_j\equiv 1\pmod{l}$,
there exist at most five primes $q_i\not\equiv 1\pmod{l}$ such that $(q_i^l-1)/(q_i-1)=p^m q_j$ for any prime $l\geq 59$
and at most six such primes for each prime $19\leq l\leq 53$.
\end{lem}

In Section \ref{pr}, we solve this diophantine problem to prove the theorem.
Here, we avoid the use of Baker's method by adopting a linear algebraic technique used by Beukers in \cite{Beu},
who gave upper bounds for the numbers of solutions of generalized Ramanujan-Nagell equations.

\section{Preliminaries}\label{lemmas}

In this section, we shall introduce some notations and lemmas.

We begin by introducing two well-known lemmas concerning prime factors
of the $n$-th cyclotomic polynomial, which we denote by $\Phi_n(X)$.
Lemma \ref{lm1} follows from Theorems 94 and 95 in Nagell \cite{Nag}.
Lemma \ref{lm2} has been proved by Bang \cite{Ban} and rediscovered by many authors
such as Zsigmondy \cite{Zsi}, Dickson \cite{Dic} and Kanold \cite{Ka1, Ka2}.

\begin{lem}\label{lm1}
Let $p, q$ be distinct primes with $q\neq 2$ and $c$ be a positive integer.
If $p\equiv 1\pmod{q}$, then $q$ divides $\sigma(p^c)$ if and only if $q$ divides $c+1$.
Moreover, if $p\not\equiv 1\pmod{q}$, then
$q$ divides $\sigma(p^c)$ if and only if the multiplicative order of $q$ modulo $p$
divides $c+1$.
\end{lem}

\begin{lem}\label{lm2}
If $a$ is an integer greater than $1$, then $\Phi_n(a)$ has
a prime factor which does not divide $a^m-1$ for any $m<n$,
unless $(a, n)=(2, 1), (2, 6)$ or $n=2$ and $a+1$ is a power of $2$.
\end{lem}

Next, we need some notations and results from the arithmetic of a quadratic field.
Let $l>3$ be a prime and $D=(-1)^\frac{l-1}{2}l$.
Let $\OK$ and $\OO$ denote $\Q(\sqrt{D})$ and its ring of integers $\Z[(1+\sqrt{D})/2]$ respectively.
We use the overline symbol to express the conjugate in $\OK$.
In the case $D>0$, $\ep$ and $R=\log\ep$ shall denote the fundamental unit and the regulator in $\OK$.
In the case $D<-4$, we set $\ep=-1$ and $R=\pi i$.  We note that neither $D=-3$ nor $-4$ occurs
since we have assumed that $l>3$.

We shall introduce the following lemma on the value of the cyclotomic polynomial $\Phi_l(x)$.

\begin{lem}\label{lm3}
Assume that $l$ is a prime $\geq 19$ and $x$ is an integer $>3^{\floor{(l+1)/6}}$.  Then $\Phi_l(x)$ can be written in the form $X^2-DY^2$ for some coprime integers $X$ and $Y$
with $0.4387/x<\abs{Y/(X+Y\sqrt{D})}$ and $\abs{Y/(X-Y\sqrt{D})}<0.5608/x$.
Moreover, if $p, q$ are primes $\equiv 1\pmod{l}$ and $\Phi_l(x)=p^m q$ for some integer $m$, then,
\begin{equation}\label{eq21}
\left[\frac{X+Y\sqrt{D}}{X-Y\sqrt{D}}\right]=\left(\frac{\bar\p}{\p}\right)^{\pm m}\left(\frac{\bar\q}{\q}\right)^{\pm 1},
\end{equation}
where $[p]=\p\bar\p$ and $[q]=\q\bar\q$ are prime ideal factorizations in $\OO$.
\end{lem}

\begin{proof}
Let $\zeta$ be a primitive $l$-th root of unity.
We can factor $(x^l-1)/(x-1)=\psi^+(x)\psi^-(x)$ in $\OK$, where
\begin{equation*}
\begin{split}
\psi^+(x)=& \prod_{\left(\frac{i}{l}\right)=1}(x-\zeta^i)=\sum_{i=0}^{\frac{l-1}{2}}a_i x^{\frac{l-1}{2}-i}, \\
\psi^-(x)=& \prod_{\left(\frac{i}{l}\right)=-1}(x-\zeta^i).
\end{split}
\end{equation*}
Hence, taking $P(x)=\psi^+(x)+\psi^-(x)$ and $Q(x)=(\psi^+(x)-\psi^-(x))/\sqrt{D}$, we have
\begin{equation}
\frac{x^l-1}{x-1}=\psi^+(x)\psi^-(x)=\frac{P^2(x)-DQ^2(x)}{4}.
\end{equation}
Now, putting $X=P(x)$ and $Y=Q(x)$, we have $\Phi_l(x)=(X^2-DY^2)/4$ with
$\psi^+(x)=(X+Y\sqrt{D})/2$ and $\psi^-(x)=(X-Y\sqrt{D})/2$.

If $\Phi_l(x)=p^m q$, then
we have the ideal factorizations $[x-\zeta^i]=\p^{(i) m} \q^{(i)}$ for $i=1, 2, \ldots, l-1$
in $\Q(\zeta)$ with $[p]=\prod_{i=1}^{l-1} \p^{(i)}$ and $[q]=\prod_{i=1}^{l-1} \q^{(i)}$.
We see that
$\prod_{\left(\frac{i}{l}\right)=1}\p^{(i)}=\p$ or $\bar\p$
and $\prod_{\left(\frac{i}{l}\right)=1}\q^{(i)}=\q$ or $\bar\q$.
Now $[X+Y\sqrt{D}]$ can be factored into one of the forms $\p^m \q, \bar\p^m \q, \p^m \bar\q$ or $\bar\p^m \bar\q$ in $\OO$
and (\ref{eq21}) holds.

Now it remains to show that $0.4387/x<\abs{Y/(X+Y\sqrt{D})}$ and \\ $\abs{Y/(X-Y\sqrt{D})}<0.5608/x$.
We begin by dealing the case $x\geq l^2$.
We clearly have $a_0=1$.
It follows from the well known result for the Gauss sum that $a_1=\frac{1\pm \sqrt{D}}{2}$.
Moreover, it immediately follows from the definition of $\psi^+(x)$ that
\[\abs{a_i}\leq\binom{(l-1)/2}{i}<\left(\frac{l-1}{2}\right)^i\] for each $i\leq\frac{l-1}{2}$.
Combining these facts on $a_i$'s, we obtain
\begin{equation}
\begin{split}
\abs{P(x)-2x^\frac{l-1}{2}-x^\frac{l-3}{2}}
& \leq 2\sum_{i=2}^{\frac{l-1}{2}}\left(\frac{l-1}{2}\right)^i x^{\frac{l-1}{2}-i} \\
& <\frac{(l-1)^2 x^\frac{l-3}{2}}{2x-l-1}\leq \frac{(l-1)^2 x^\frac{l-3}{2}}{2l^2-l-1} \\
& <\frac{x^\frac{l-3}{2}}{2}
\end{split}
\end{equation}
and
\begin{equation}
\begin{split}
\abs{\abs{Q(x)}-x^\frac{l-3}{2}}
& \leq \frac{2}{\sqrt{l}}\sum_{i=2}^{\frac{l-1}{2}}\left(\frac{l-1}{2}\right)^i x^{\frac{l-1}{2}-i} \\
& <\frac{(l-1)^2 x^\frac{l-3}{2}}{\sqrt{l}(2x-l-1)}<\frac{(l-1)^2 x^\frac{l-3}{2}}{\sqrt{l}(2l^2-l-1)} \\
& <\frac{x^\frac{l-3}{2}}{2\sqrt{l}}.
\end{split}
\end{equation}
From these inequalities, we deduce that
\begin{equation}
\abs{\frac{Q(x)}{P(x)-Q(x)\sqrt{D}}}<\frac{1+\frac{1}{2\sqrt{l}}}{2x+\frac{1}{2}-\left(1+\frac{1}{2\sqrt{l}}\right)\sqrt{l}}<\frac{0.5608}{x}
\end{equation}
and
\begin{equation}
\abs{\frac{Q(x)}{P(x)+Q(x)\sqrt{D}}}>\frac{1-\frac{1}{2\sqrt{l}}}{2x+\frac{3}{2}+\left(1+\frac{1}{2\sqrt{l}}\right)\sqrt{l}}>\frac{0.4387}{x}
\end{equation}
for $l\geq 19$, proving the lemma in this case.

In the remaining case $x<l^2$, then we have $l\leq 37$ since we have assumed that $x>3^{\floor{(l+1)/6}}$.
For each $l$, we can confirm the desired inequality for $3^{\floor{(l+1)/6}}<x<l^2$ by calculation.
Now the lemma is completely proved.
\end{proof}

\section{Reduction to a diophantine problem}\label{reduction}

Let $N=p^{\alpha} q_1^{2\beta}\cdots q_r^{2\beta}$ be an odd perfect number.
In this section, we shall show that our theorem can be reduced to Lemma \ref{lm0}.

Various results referred in the introduction of this paper allows us to assume that $\beta\geq 9$
without loss of generality.

We see that we can take a prime factor $l$ of $2\beta+1$ which is one of the $q_i$'s.
Indeed, if $2\beta+1$ has at least two distinct prime factors $l_1$ and $l_2$,
then at least one of them must be one of the $q_i$'s and,
if $2\beta+1=l^\gamma$ is a power of a prime $l$, then
we must have $l=q_{i_0}$ for some $i_0$ by Kanold \cite{Ka1}.

As we did in \cite{Ymd}, we divide $q_1, \ldots, q_r$ into four disjoint sets. Let
\begin{equation*}
S=\{i:q_i\equiv 1\pmod{l}\},
\end{equation*}
\begin{equation*}
T=\{i:q_i\nequiv 1\pmod{l}, i\neq i_0, q_j\mid\sigma(q_i^{2\beta})\text{ for some }1\leq j\leq r\},
\end{equation*}
and
\begin{equation*}
U=\{i:q_i\nequiv 1\pmod{l}, i\neq i_0, q_j\nmid\sigma(q_i^{2\beta})\text{ for any }1\leq j\leq r\}.
\end{equation*}
Hence, we can write $\{i: 1\leq i\leq r\}=S\cup T\cup U\cup \{i_0\}$.

In \cite{Ymd}, we proved that $\#S\leq 2\beta$.
Moreover, if $2\beta+1=l^\gamma$ is a prime power, then $\#T\leq (2\beta)^2$ and $\#U\leq 1$,
implying that $r\leq 4\beta^2+2\beta+2$
and, if $2\beta+1$ has $s>1$ distinct prime factors, then $\#S\leq 2\beta$ and $r\leq 2\beta\#S/(2^{s-1}-1)$.

For each $i\in T$, let $f(i)$ denote the number of prime factors in $S$ dividing $\sigma(q_i^{2\beta})$
counted with multiplicity.  Then, we can easily see that, for any $i\in T$, $\sigma(q_i^{2\beta})$ has at least one prime factor in $S$ from Lemmas \ref{lm1} and \ref{lm2}.
Hence, we have $f(i)\geq 1$ for any $i\in T$.

This immediately gives that $\#T\leq \sum_{i\in T}f(i)\leq (2\beta)\#S\leq 4\beta^2$,
which is Lemma 3.2 in \cite{Ymd}.  This yields the dominant term $4\beta^2$ in the exponent of the classical bound, which we would like to improve.
To this end, we denote by $\delta$ the number of $i$'s for which $f(i)=1$.
Then we have \[2\#T-\delta=\delta+2(\#T-\delta)\leq \sum_{i\in T}f(i)\leq 4\beta^2;\]
that is, $\#T\leq 2\beta^2+(\delta/2)$.

If $2\beta+1$ is composite, then, by Lemma \ref{lm2}, for each divisor $d$ of $(2\beta+1)/l$,
$\Phi_{ld}(q_i)$ has a prime factor $\equiv 1\pmod{l}$ not dividing $\Phi_{lk}(q_i)$ for any other divisor $k<d$ of $(2\beta+1)/l$.
Hence, we see that $U$ must be empty.
Moreover, if $i\in T$ and $f(i)=1$, then
$2\beta+1=l^2$ and $\Phi_l(q_i)=q_j$ or $\Phi_{l^2}(q_i)=q_j$ for some $j\in S$,
or $2\beta+1=l_1 l$ for some prime $l_1$ and $\Phi_l(q_i)=q_j$ or $\Phi_{l_1 l}(q_i)=q_j$ for some $j\in S$.
From this, we can deduce that $f(i)=1$ holds for at most $2\#S\leq 4\beta$ indices $i\in T$.
That is, $\delta\leq 2\#S\leq 4\beta$.  Since $U$ is empty, we have $\#T+\#U\leq 2\beta^2+2\beta$.

If $2\beta+1=l$ is prime, $i\in T$ and $f(i)=1$,
then $\sigma(q_i^{2\beta})=\Phi_l(q_i)=p^m q_j$ for an index $j\in S$ and an integer $m\geq 0$.
Moreover, we have $\#U\leq 1$ as mentioned above.

Now, observing that $r\leq \#S+\#T+\#U+1\leq 2\beta^2+2\beta+(\delta/2)+2$,
we conclude that Theorem \ref{th} can be derived if we show that,
for each prime $q_j\equiv 1\pmod{l}$,
there exist at most five primes $q_i$ with $i\in T$ such that $\Phi_l(q_i)=(q_i^l-1)/(q_i-1)=p^m q_j$ for any prime $l\geq 59$
and at most six such primes for each prime $19\leq l\leq 53$.
Hence, Theorem \ref{th} would follow from Lemma \ref{lm0}, which we prove in the next section.

\section{Proof of the theorem}\label{pr}

We begin by proving a gap principle using elementary modular arithmetic.
\begin{lem}\label{lm4}
If $x_2>x_1>0$ are two multiplicatively independent integers and $\Phi_l(x_1)=p^{m_1} q_j$ and $\Phi_l(x_2)=p^{m_2} q_j$,
then $x_2>x_1^{\floor{(l+1)/6}}$.
\end{lem}

\begin{proof}
Assume that $x_2\leq x_1^{\floor{(l+1)/6}}$.
We begin by observing that $(x_1^{f_1} x_2^{f_2})^l\equiv 1\pmod{q_j}$ for any integers $f_1$ and $f_2$.
In the case $p^{m_1}<q_j$, we must have $q_j>(\Phi_l(x_1))^{1/2}>x_1^{(l-1)/2}$
and therefore
\begin{equation}
1\leq x_1^{f_1}x_2^{f_2}\leq x_1^{f_1+f_2\floor{(l+1)/6}}\leq x_1^{(l-1)/2}<q_j
\end{equation}
for $0\leq f_1\leq (l-1)/2-f_2\floor{(l+1)/6}$.
This implies that each integer of the form $x_1^{f_1} x_2^{f_2}$
with $0\leq f_1\leq (l-1)/2-f_2\floor{(l+1)/6}$ must give a solution of the congruence
$X^l\equiv 1\pmod{q_j}$ and these solutions are not congruent to each other.  For each fixed $f_2$,
we have $(l+1)/2-f_2\floor{(l+1)/6}$ such solutions.
Hence, recalling that $l\geq 19$,
the congruence $X^l\equiv 1\pmod{q_j}$ should have at least
\[\sum_{f_2=0}^2\left(\frac{l+1}{2}-f_2\floor{\frac{l+1}{6}}\right)=\frac{3(l+1)}{2}-3\floor{\frac{l+1}{6}}\geq l+1\]
solutions in $1\leq X<q_j$, which is impossible.  Similarly, in the case $p^{m_1}>q_j$,
the congruence $X^l\equiv 1\pmod{p^{m_1}}$ should have at least $l+1$ solutions in $1\leq X<p^{m_1}$,
a contradiction again.  Hence, we must have $x_2>x_1^{\floor{(l+1)/6}}$.
\end{proof}

Using Lemma \ref{lm3}, we shall prove another gap principle, which is more conditional
but much stronger than the first gap principle.

\begin{lem}\label{lm5}
If $\Phi_l(x_i)=p^{m_i} q_j$ for three integers $x_3>x_2>x_1>0$ with $x_2>x_1^{\floor{(l+1)/6}}$, then $m_3>0.445\abs{R}x_1/\sqrt{l}$.
\end{lem}
\begin{proof}
We write $\xi_i=(X_i+Y_i\sqrt{D})/(X_i-Y_i\sqrt{D})$ for each $i=1, 2, 3$.
Factoring $[p]=\p\bar\p, [q_j]=\q_j\bar\q_j$ in $\OO$ and applying Lemma \ref{lm3} with $q_j$ in place of $q$,
we obtain that, for each $i=1, 2, 3$,
\begin{equation}
[\xi_i]=\left(\frac{\bar\p}{\p}\right)^{\pm m_i}\left(\frac{\bar\q_j}{\q_j}\right)^{\pm 1},
\end{equation}
holds with $0<Y_i/(X_i-Y_i\sqrt{D})<(\Phi_l(x_i))^{-1/(l-1)}$.

Hence, taking an appropriate combination of signs, we obtain
\begin{equation}
[\xi_1]^{\pm m_2\pm m_3}[\xi_2]^{\pm m_3\pm m_1}[\xi_3]^{\pm m_1\pm m_2}=[1],
\end{equation}
and therefore
\begin{equation}
\xi_1^{\pm m_2\pm m_3}\xi_2^{\pm m_3\pm m_1}\xi_3^{\pm m_1\pm m_2}=\pm \ep^a
\end{equation}
for some integer $a$.
Hence, if we let each logarithm $\log\xi_i$ take its principal value, we have
\begin{equation}
(\pm m_2\pm m_3)\log\xi_1+(\pm m_3\pm m_1)\log\xi_2+(\pm m_1\pm m_2)\log\xi_3=bR
\end{equation}
for some integer $b$.

If $b\neq 0$, then
\begin{equation}
(m_2+m_3)\abs{\log \xi_1}+(m_3+m_1)\abs{\log \xi_2}+(m_1+m_2)\abs{\log \xi_3}\geq \abs{R}.
\end{equation}
Recalling that $0<Y_i/(X_i-Y_i\sqrt{D})<0.5608/x_i$ from Lemma \ref{lm3}
and each complex logarithm takes its principal value,
we have
\[\log\frac{X_i+Y_i\sqrt{D}}{X_i-Y_i\sqrt{D}}=\log \left(1+\frac{2Y_i\sqrt{l}}{X_i-Y_i\sqrt{l})}\right)<\frac{2Y_i\sqrt{l}}{X_i-Y_i\sqrt{l}}<\frac{1.1216\sqrt{l}}{x_i}\]
for $D>0$ and
\[\log\frac{X_i+Y_i\sqrt{D}}{X_i-Y_i\sqrt{D}}=2\arctan \frac{Y_i\sqrt{l}}{X_i}<\frac{2Y_i\sqrt{l}}{X_i-Y_i\sqrt{l}}<\frac{1.1216\sqrt{l}}{x_i}\]
for $D<0$, we have $\abs{\log \xi_i}<1.1216\sqrt{l}/x_i$ whether $D>0$ or $D<0$.
Hence, we have $\abs{\log \xi_i}<1.1216\sqrt{l}/x_i$ and therefore
\begin{equation}
\begin{split}
& 2.2432m_3\sqrt{l}\left(\frac{1}{x_1}+\frac{1}{x_2}+\frac{1}{x_3}\right) \\
& \geq 1.1216\sqrt{l}\left(\frac{m_2+m_3}{x_1}+\frac{m_3+m_1}{x_2}+\frac{m_1+m_2}{x_3}\right)>\abs{R}.
\end{split}
\end{equation}
From this and the assumption that $x_3>x_2>x_1^{\floor{(l+1)/6}}\geq x_1^3$
(recall that we have assumed that $l\geq 19$), we can deduce that $m_3>0.445x_1 \abs{R}/\sqrt{l}$.

If $b=0$, then
$\abs{\log \xi_1}\leq 2m_3\abs{\log \xi_2}+2m_2\abs{\log\xi_3}$.
We see that $\abs{\log \xi_2}<1.1216/x_2, \abs{\log \xi_3}<1.1216/x_3$ by Lemma \ref{lm3} and
$\abs{\log \xi_1}>0.8774/x_1$.
Hence, we have
\[\frac{0.15}{x_1}<m_3\left(\frac{1}{x_2}+\frac{1}{x_3}\right).\]

Moreover, since $x_2>x_1^{\floor{(l+1)/6}}\geq x_1^3$ and $x_3>x_2^{\floor{(l+1)/6}}$, we have
\begin{equation}
m_3>\frac{0.15x_2}{x_1}>0.15x_1^2>0.15\times 3^{\floor{\frac{l+1}{6}}} x_1>\abs{R} x_1,
\end{equation}
where the last inequality follows observing that,
if $l\equiv 3\pmod{4}$, then $0.15\times 3^{\floor{(l+1)/6}}>0.2l>\pi=\abs{R}$
and, if $l\equiv 1\pmod{4}$, then $D=l\geq 29$ and therefore $0.15\times 3^{\floor{(l+1)/6}}>l>R$
using the estimate $R<D^{1/2}\log (4D)$ from \cite{Fai}.

Hence, we conclude that, whether $b=0$ or not, $m_3>0.445\abs{R}x_1/\sqrt{l}$, proving the lemma.
\end{proof}

Now we shall prove Lemma \ref{lm0}.
Fix a prime $q_j\equiv 1\pmod{l}$.
Assume that $q_1<q_2<\cdots <q_6$ are six primes not congruent to $1$ modulo $l$ such that
$\Phi_l(q_i)=p^{g_i} q_j$ for $i=1, 2, \ldots, 6$.
Moreover, assume that $q_7$ is a prime in $T$ greater than $q_{10}$
and $\Phi_l(q_7)=p^{g_7} q_j$ if $19\leq l\leq 53$.
Write $R^\prime=0.445\abs{R}$.
Since $q_2>q_1^{\floor{(l+1)/6}}\geq 3^{\floor{(l+1)/6}}$,
we can apply Lemma \ref{lm5} with
$(x_i, m_i)=(q_{i+1}, g_{i+1}) (i=1, 2, 3)$
to obtain \[\log q_4>\frac{g_4\log p}{l-1}> \frac{q_2 R^\prime \log p}{(l-1)\sqrt{l}
}
\geq \frac{3^{\floor{\frac{l+1}{6}}} R^\prime \log (2l+1)}{(l-1)\sqrt{l}},\]
where we use the fact $p\geq 2l+1$ by Lemma \ref{lm1} and the assumption $q_i\not\equiv 1\pmod{l}$.
Similarly, we have $\log q_6>q_4 R^\prime (\log (2l+1))/(l-1)\sqrt{l}$
and
\begin{equation}
\frac{\log q_6}{\log 2}>\exp\left(\frac{3^{\floor{\frac{l+1}{6}}}R^\prime \log (2l+1)}{(l-1)\sqrt{l}}+\log\frac{R^\prime \log (2l+1)}{(l-1)\sqrt{l}}\right).
\end{equation}

If $l\geq 79$ and $l\equiv 3\pmod{4}$, then $\abs{R}=\pi$
and
\begin{equation}\label{eq41}
\frac{\log q_6}{\log 2}>4^{l^2}=4^{4\beta^2+4\beta+1}.
\end{equation}
If $l\geq 73$ and $l\equiv 1\pmod{4}$, then, observing that $R>\sqrt{l}$,
we have (\ref{eq41}) again.
Similarly, if $l=61$, then we have (\ref{eq41}) observing that $R=(39+5\sqrt{61})/2$.

Assume that $l=71$.
If $q_1=3$, then $q_j=\Phi_{71}(3)$.
Since $q_1$ and $q_2$ are multiplicatively independent, we must have $q_2>\Phi_{71}(3)$.
If $q_1\geq 5$, then $q_2>5^{12}$.
In both cases, $q_2>5^{12}$ and $\log q_6/\log 2>4^{5041}$.
If $l=67$, then we must have $q_1\geq 17, q_2\geq 17^{11}$ and $\log q_6/\log 2>4^{4489}$.
If $l=59$, then $q_1\geq 5$ since $\Phi_{59}(3)=14425532687\times 489769993189671059$,
where both prime factors are congruent to $3$ modulo $4$.
Hence, we have $q_2>5^{10}$ and $\log q_6/\log 2>4^{2809}$.

Hence, for $l\geq 59$, we have (\ref{eq41}),
which is impossible since it implies that $N\geq q_6$ exceeding the classical bound.

If $23\leq l\leq 53$, applying Lemma \ref{lm4}, we have $q_1\geq 3, q_2\geq 3^4$ and $q_3\geq 3^{16}$.
Applying Lemma \ref{lm5} with $(x_i, m_i)=(q_{i+2}, g_{i+2})$
$(i=1, 2, 3)$ and then $(x_i, m_i)=(q_{i+4}, g_{i+4}) (i=1, 2, 3)$, we have
$q_5>\exp(610000)$ and $q_7>\exp(\exp(600000))$.
If $l=19$, applying Lemma \ref{lm5} with $(x_i, m_i)=(q_{i+2}, g_{i+2}) (i=1, 2, 3)$
and then $(x_i, m_i)=(q_{i+4}, g_{i+4}) (i=1, 2, 3)$, we have
$q_1\geq 3, q_2\geq 29, q_3\geq 24391, q_5>\exp(2200)$ and $q_7>\exp(\exp(2000))$.
Thus, $q_7$ must exceed the classical bound if $19\leq l\leq 53$,
which is a contradiction again.

Hence, we conclude that, for each given $j\in S$,
there are at most five indices $i\in T$ with $f(i)=1$ and $q_j\mid \Phi_l(q_i)=\sigma(q_i^{2\beta})$ if $l\geq 59$
and there are at most six indices $i\in T$ and $q_j\mid \Phi_l(q_i)=\sigma(q_i^{2\beta})$ with $f(i)=1$ if $19\leq l\leq 53$.
This completes the proof of Lemma \ref{lm0}, which in turn implies Theorem \ref{th}.

{}
\vskip 12pt

{\small Tomohiro Yamada}\\
{\small Center for Japanese language and culture\\Osaka University\\562-8558\\8-1-1, Aomatanihigashi, Minoo, Osaka\\Japan}\\
{\small e-mail: \protect\normalfont\ttfamily{tyamada1093@gmail.com}}
\end{document}